\documentclass[11pt, a4paper,leqno]{amsart}
\usepackage{amsmath,amsthm,amscd,amssymb,amsfonts, amsbsy}
\usepackage{latexsym}
\usepackage{txfonts}
\usepackage{exscale}

\usepackage[colorlinks,citecolor=red, linkcolor=blue, hypertexnames=false]{hyperref}

\usepackage{tikz}
\usepackage{tkz-euclide}
\usetkzobj{all}
\usetikzlibrary{intersections}
\usetikzlibrary{shapes,arrows,calc}
\usetikzlibrary{shapes.geometric}

\usepackage{color}

\allowdisplaybreaks

\theoremstyle{plain}
\newtheorem{theorem}[equation]{Theorem}
\newtheorem{lemma}[equation]{Lemma}

\theoremstyle{definition}
\newtheorem{definition}[equation]{Definition}

\newtheorem{example}[equation]{Example}

\theoremstyle{remark}
\newtheorem{remark}[equation]{Remark}

\newtheorem{claim}[equation]{Claim}

\numberwithin{equation}{section}

\newcommand{\cp}{{\rm cap}}

\newcommand{\RR}{{\mathbb{R}}}
\newcommand{\ZZ}{{\mathbb{Z}}}

\newcommand{\dist}{\operatorname{dist}}

\newcommand{\re}{\mathbb{R}}
\newcommand{\rn}{\mathbb{R}^n}

\newcommand{\reu}{\mathbb{R}^{n+1}_+}
\newcommand{\ree}{\mathbb{R}^{n+1}}
\newcommand{\N}{\mathbb{N}}
\newcommand{\dd}{\mathbb{D}}

\newcommand{\om}{\Omega}

\newcommand{\E}{\mathcal{E}}

\newcommand{\pom}{\partial\Omega}

\newcommand{\hm}{\omega}

\newcommand{\mres}[1]{|_{#1}}

\renewcommand{\emptyset}{\mbox{\textup{\O}}}

\DeclareMathOperator{\supp}{supp}

\DeclareMathOperator{\diam}{diam}

\DeclareMathOperator{\ang}{angle}

\begin{document}
\allowdisplaybreaks

\title[Rectifiability, interior approximation, and Harmonic Measure]{Rectifiability, interior approximation and Harmonic Measure}

\author{Murat Akman}
\address{Murat Akman
\\
Department of Mathematics 
\\
University of Connecticut
\\ 
Storrs CT 06269, USA}
\email{murat.akman@uconn.edu}

\author{Simon Bortz}

\address{Simon Bortz
\\
Department of Mathematics
\\
University of Minnesota
\\
Minneapolis, MN, USA} 
\email{bortz010@umn.edu}

\author{Steve Hofmann}

\address{Steve Hofmann
\\
Department of Mathematics
\\
University of Missouri
\\
Columbia, MO 65211, USA} \email{hofmanns@missouri.edu}

\author{Jos\'e Mar{\'\i}a Martell}

\address{Jos\'e Mar{\'\i}a Martell
\\
Instituto de Ciencias Matem\'aticas CSIC-UAM-UC3M-UCM
\\
Consejo Superior de Investigaciones Cient{\'\i}ficas
\\
C/ Nicol\'as Cabrera, 13-15
\\
E-28049 Madrid, Spain}
\email{chema.martell@icmat.es}

\date{June 1, 2017. \textit{Revised}: \today.}
\subjclass[2010]{31A15, 30C85, 42B37, 31B05, 28A75, 28A78, 49Q15} 

\keywords{Harmonic measure, rectifiability}

\thanks{The first and last authors acknowledge financial   support from the Spanish Ministry of Economy and Competitiveness,   through the ``Severo Ochoa'' Programme for Centres of Excellence in
  R\&D” (SEV-2015-0554). They also acknowledge that the research   leading to these results has received funding from the European   Research Council under the European Union's Seventh Framework
  Programme (FP7/2007-2013)/ ERC agreement no. 615112 HAPDEGMT. The second and third authors were supported by NSF grant DMS-1361701. The last author would like to express his gratitude to the University of Missouri-Columbia (USA), for its support and hospitality while he was visiting this institution.
\\
All authors wish to thank Matthew Badger, Svitlana Mayboroda,  and Tatiana Toro for their helpful comments and suggestions.}

\begin{abstract}

We prove a structure theorem for any $n$-rectifiable set $E\subset \mathbb{R}^{n+1}$, $n\ge 1$, satisfying a weak version of the lower ADR condition, and having locally finite $H^n$ ($n$-dimensional Hausdorff) measure. Namely, that $H^n$-almost all of $E$ can be covered by a countable union of boundaries of bounded Lipschitz domains contained in $\mathbb{R}^{n+1}\setminus E$.  As a consequence, for harmonic measure in the complement of such a set $E$,   we establish
a non-degeneracy condition  which amounts to 
saying that $H^n|_E$ is ``absolutely continuous" with respect to 
harmonic measure in the sense that any Borel subset of $E$ with strictly 
positive $H^n$ measure has strictly positive 
harmonic measure in some connected component of 
$\mathbb{R}^{n+1}\setminus E$. We also provide some counterexamples showing that our result for harmonic measure is optimal. Moreover, we show that if, in addition, a set $E$ as above is the boundary of a connected domain $\Omega \subset \mathbb{R}^{n+1}$ which satisfies an infinitesimal interior thickness condition,  then $H^n|_{\partial\Omega}$ is 
absolutely continuous (in the usual sense) with respect to  harmonic measure for $\Omega$. Local versions of these results are also proved: if just some piece of the boundary is $n$-rectifiable then we get the corresponding absolute continuity on that piece. As a consequence of this and recent results in \cite{AHM3TV}, we can decompose the boundary of any open connected set satisfying the previous conditions in two disjoint pieces: one that is $n$-rectifiable where Hausdorff measure is absolutely continuous with respect to harmonic measure and another purely $n$-unrectifiable piece  having vanishing harmonic measure.

\end{abstract}

\maketitle


\section{Introduction}
The connection between regularity of the boundary and properties of 
harmonic measure for a domain has been studied extensively; we recall a 
few relevant results. In \cite{Rfm} it was shown that if $\om \subset \re^2$ is simply 
connected with rectifiable boundary, then arc-length measure $\sigma$
and harmonic measure $\hm$ are 
mutually absolutely continuous. 
In contrast to the simply connected case, in \cite{BiJo} it was shown that there 
exists a domain in $\re^2$ which is the complement of a (uniformly) $1$-rectifiable set, 
for which $\hm$ fails to be absolutely continuous with respect to $\sigma$. 
A quantitative version of the result of \cite{Rfm}
was obtained in \cite{Lav}. 
In higher dimensions, it was shown that for Lipschitz domains \cite{D1}, and chord 
arc domains \cite{DJe},  harmonic measure and surface measure are quantitatively 
mutually absolutely continuous (in the sense of the Muckenhoupt $A_\infty$ condition). On the other hand, we know that the analogue of \cite{Rfm} fails to hold in higher dimensions. Wu \cite{W} and Ziemer \cite{Z} produced examples of topological 2-spheres in $\re^3$ with locally finite perimeter in which harmonic measure is fails to be absolutely continuous with respect to surface measure and surface measure fails to be absolutely continuous with respect to harmonic measure respectively. More recently, in \cite{Ba}, the author proved that surface measure $\sigma$ is absolutely continuous
with respect to $\hm$ in an NTA domain 
$\Omega$ with locally finite perimeter, thus replacing
the upper Ahlfors-David regularity (``ADR") condition on $\pom$ assumed in \cite{DJe} by a weaker
qualitative condition, namely, local finiteness of $\sigma$ 
(the lower ADR bound holds automatically for NTA domains, by the local isoperimetric inequality).
A refinement of the result in \cite{Ba} was obtained in \cite{Mo}, where it is shown
that for a uniform domain of locally finite perimeter, with rectifiable boundary 
satisfying the lower ADR condition,  surface measure is again absolutely continuous with respect to 
harmonic measure. Independently, \cite{ABHM} obtained this result (as well as its converse) assuming ``full''
(i.e., upper and lower) ADR.

Let us point out that in all of the results just mentioned (aside from the counter-example constructed in \cite{BiJo})
there is some strong connectivity hypothesis
(i.e., simple connectivity or the Harnack chain condition), 
and in higher dimensions a special quantitative
openness condition (the so-called ``corkscrew" condition).   Furthermore, in light of the 
Bishop-Jones example, 
{\it strong} connectivity of some sort seems to be necessary to 
obtain absolute continuity of 
harmonic measure with respect to surface measure.  Indeed, the 
Bishop-Jones domain itself is 
connected\footnote{Of course it is not simply connected, nor does it satisfy the 
Harnack chain condition.}, satisfies an interior
corkscrew condition,
and has a uniformly rectifiable boundary (in particular, arclength measure on the boundary is Ahlfors-David regular),
yet harmonic measure has positive mass on a set of arclength measure zero.
By contrast, in this paper we show, for a large class of open sets in 
$\ree, \,n\geq 1$, not necessarily connected,
with rectifiable boundaries and locally finite perimeter,
that harmonic measure cannot vanish on a set of positive surface measure.
More precisely,
in our main result, Theorem \ref{T1}, we show 
that if $E$ is a closed $n$-rectifiable set
satisfying some weak local version of the lower ADR condition, 
and on which
Hausdorff $H^n$ measure is locally finite, then the
surface measure $\sigma:=H^n|_E$ is ``absolutely  continuous" with 
respect to harmonic measure for
$\ree\setminus E$ in the sense
that any Borel subset of $E$ with positive surface 
measure has non-zero harmonic measure 
in at least one of the connected components of $\ree\setminus E$. 
Assuming in addition that $E=\pom$ is the boundary of a connected open set
$\Omega$ satisfying a weak version of an interior corkscrew condition, we prove in Theorem
\ref{T3} that $\sigma$ is absolutely continuous (in the usual sense) 
with respect to harmonic measure
for $\Omega$.   The weak corkscrew condition of Theorem \ref{T3} is an
 ``interior thickness'' condition 
which guarantees that at infinitesimal scales any ball centered at the boundary captures 
a non-degenerate portion of the set.
In particular, the 
domain constructed in \cite{BiJo},
for which harmonic measure fails to be absolutely continuous with respect to
$\sigma$, 
nonetheless has the property that $\sigma$ is absolutely 
continuous with respect to harmonic measure.

The proof of Theorem \ref{T1} 
relies on a structure theorem
(Theorem \ref{T2}), which allows us to
cover $H^n$-almost all of $E$ by a countable 
union of boundaries of bounded Lipschitz 
domains contained in $\ree\setminus E$.  A similar structure result
is involved in the proof of Theorem \ref{T3}.
The novelty of these structure results is of 
course the fact that the Lipschitz domains are subdomains of
$\ree\setminus E$ (or of $\Omega$ in the case of Theorem \ref{T3}), 
since $n$-rectifiability already entails coverage $H^n$-a.e. by Lipschitz graphs. 
This approximability by Lipschitz subdomains allows one to
use the maximum principle along with
Dahlberg's Theorem \cite{D1} to obtain the conclusions of Theorems \ref{T1} and \ref{T3}.
We note that the proofs in \cite{DJe} and \cite{Ba} (see also \cite{Azzam}) 
are also based on constructive approximation by Lipschitz subdomains, 
so in some sense the present paper may be viewed
as a qualitative version of those works.  Let us mention in addition that our methods have much in common
with the proof of McMillan's Theorem given in
 \cite[pp 207-210]{GM}.  Somewhat more precisely,
McMillan's Theorem says that for a simply connected domain
$\Omega \subset \RR^2$,  the set $K$ of ``cone points" of $\pom$ is rectifiable, and
harmonic measure $\hm$ and arc-length measure $\sigma$ are mutually absolutely continuous on $K$.
Here, $x\in \pom$ is a cone point if there is a truncated open cone $\Gamma$ with vertex
at $x$, such that $\Gamma \subset \Omega$.  Although simple connectivity is used 
strongly to establish the direction $\hm \ll \sigma$ in McMillan's theorem, the proof also 
contains an implicit structure
theorem for the cone set $K$, which does not really require simple connectivity.  This structure theorem
allows one to construct an open subset\footnote{The open set $\Omega'$ is 
a simply connected domain in the case that $\Omega$ is simply connected.}
$\Omega'\subset \Omega$, with a rectifiable boundary
such that $\pom'\cap\pom = E$, for any $E\subset K$.   Our structure theorem in higher dimensions
says that in the presence of our background hypotheses (including rectifiability of $E$), then
$\sigma$-a.e. point on $\pom$ is a cone point, and moreover, the cone set may be covered by
the union of boundaries of a countable collection of Lipschitz subdomains of $\ree\setminus E$.

In Section \ref{sec:counter} we present two examples of rectifiable sets which 
fail to satisfy either the locally finite perimeter or the local 
lower ADR assumptions, and for which surface measure is not absolutely continuous with 
respect to harmonic measure. 

Finally, in Appendix \ref{appendix}, we present some local versions of the previous results where absolute continuity holds in the rectifiable portions of $\pom$.  As an immediate  consequence of this and \cite{AHM3TV} we obtain that for any connected set whose boundary has $H^n$-locally finite measure and satisfying the mentioned weak lower ADR and ``interior thickness'' conditions, 
one can decompose its boundary in a good and a bad piece. The good piece is $n$-rectifiable and  Hausdorff measure is absolutely continuous with respect to harmonic measure. The bad piece is purely $n$-unrectifiable, and has vanishing harmonic measure.

\section{Main Results}

We now state our main result which gives that surface measure is absolutely continuous with respect to harmonic measure provided the set has locally finite surface measure, satisfies a weak  lower ADR condition and it is $n$-rectifiable (see Section \ref{sect:prelim} for the precise definitions):

\begin{theorem}\label{T1} 
Let $E \subset \ree$, $n\ge 1$, be a closed set with locally finite $H^n$-measure satisfying the ``weak lower ADR'' (WLADR) condition (see Definition \ref{d2} below).
Under these background hypotheses, if $E$ is $n$-rectifiable  (cf. Definition \ref{d1}) 
then $H^n|_E$ is ``absolutely continuous" with respect to harmonic measure 
for $\ree\setminus E$, in the sense that if $
F\subset E$ is a Borel set with 
$H^n(F)>0$, then $\omega^X(F)>0$ for some $X\in\ree\setminus E$.
\end{theorem}

\begin{remark}
Let us note that in the previous result 
the ``absolute continuity'' property needs to be interpreted 
properly, as we are comparing one measure $\sigma$ with the collection of harmonic measures $\{\omega^X\}_{X\in\ree\setminus E}$. An equivalent formulation of the conclusion 
is that if $F\subset E$ is a Borel set with $\omega^X(F)=0$ for every 
$X\in\ree\setminus E$, then 
necessarily $H^n(F)=0$.   
One can restate this in terms of  genuine absolute continuity of
$H^n|_E$ with respect to an averaged harmonic measure: 
$$
H^n|_E\,\ll\, \widetilde{\omega}\,:=\,\sum_{k\ge 1} 2^{-k}\,\omega_k,
$$
where $\omega_k=\omega_{D_k}^{X_k}$ is the harmonic measure for the domain $D_k$ with some fixed pole $X_k\in D_k$, and $\{D_k\}_{k\ge 1}$ is an enumeration of the connected components of $\ree\setminus E$.
\end{remark}

Our main result will follow easily from the following structural theorem which says that under the same background hypotheses we can cover $E$ by boundaries of Lipschitz subdomains of $\ree\setminus E$.

\begin{theorem}\label{T2} 
Let $E \subset \ree$, $n\ge 1$, be a closed set with locally finite $H^n$-measure satisfying the WLADR condition. Then,  $E$ is $n$-rectifiable if and only if there exists a countable collection $\{\om_j\}_j$ of bounded Lipschitz domains with $\om_j \subset \ree \setminus E$ for every $j$, and a set $Z\subset E$   with $H^n(Z)=0$ such that
\begin{equation}\label{eq3}
E \subset Z \cup \Big(\bigcup_j \pom_j\Big).
\end{equation}
\end{theorem}

As mentioned above, the 
innovation in Theorem \ref{T2} is  the fact that 
each $\Omega_j$ is contained in $\ree\setminus E$, otherwise 
this would be the standard covering of a rectifiable set 
by Lipschitz graphs.   
Theorem \ref{T1} will
follow almost directly from Theorem \ref{T2} 
and Dahlberg's Theorem for Lipschitz domains (Theorem \ref{L3}),  
by the maximum principle. 
Additionally, one may view Theorem \ref{T2} as 
a qualitative version of the results in \cite{BH1}.

Our next results deals with the case on which one starts with a domain $\Omega$ and seeks to approximate its boundary by interior Lipschitz subdomains. This in particular leads to obtain that surface measure is absolutely continuous with respect to harmonic measure for $\Omega$.

\begin{theorem}\label{T3}
Let $\Omega\subset\ree$, $n\ge 1$, be an open connected set, whose boundary $\pom$ has locally finite $H^n$-measure. Assume that $\pom$  satisfies the WLADR condition (see  Definition \ref{d2} below). Assume further that $H^n(\pom\setminus\partial_+\om)=0$ where $\partial_+\om$ is the Interior Measure Theoretic Boundary (cf. Definition \ref{d3}). 
Then,  $\pom$ is $n$-rectifiable   if and only if there exists a countable collection $\{\om_j^{\rm int}\}_j$ of bounded Lipschitz domains with $\om_j^{\rm int} \subset \Omega$ for every $j$, and a set $Z\subset \pom$  with $H^n(Z)=0$ such that
\begin{equation}\label{eq3:int}
\pom \subset Z \cup \Big(\bigcup_j \pom_j^{\rm int}\Big).
\end{equation}
As a consequence, if $\pom$ is $n$-rectifiable (and $\om$ satisfies the background hypothesis above)$H^n|_{\pom}$ is absolutely continuous with respect to $\hm$, where $\hm=\hm^X$ is the harmonic measure for $\om$ with some (or any) fixed pole  $X\in\Omega$.
\end{theorem}

\begin{remark}
The connectivity assumption here is merely cosmetic.  If $\om$ were an open set rather than a domain the conclusion would be that $H^n|_{\pom}$ is absolutely continuous with respect to $\hm$ in the sense that if $F\subset \pom$ is a Borel set with $H^n(F)>0$ then $\omega^X(F)>0$ for some $X\in \Omega$ (or any $X$ in the same connected component).
\end{remark}

\begin{remark}

Note that \eqref{eq3:int} implies that  $\pom\setminus\partial_+\om\subset Z$ and hence 
condition $\sigma(\pom\setminus \partial_+\om) = 0$ is necessary for the approximation of $\Omega$ by interior Lipschitz subdomains.
\end{remark}

As a corollary of Theorem \ref{T3} and the results in \cite{AHM3TV} we have the following characterization of $n$-rectifiability in terms of properties of harmonic measure. 

\begin{theorem}\label{T5}
Let $\Omega\subset\ree$, $n\ge 1$, be an open connected set, whose boundary $\pom$ has locally finite $H^n$-measure. Assume that $\pom$  satisfies the WLADR condition and that the Interior Measure Theoretic Boundary has full $H^n$-measure.
Then $\pom$ is $n$-rectifiable if and only if $H^n|_{\pom}$ is absolutely continuous with respect to $\hm$, where $\hm=\hm^X$ is the harmonic measure for $\om$ with some (or any) fixed pole  $X\in\Omega$.
\end{theorem}

Let us point out that this equivalence has been shown in \cite{ABHM} in $\ree$, $n\ge 2$, under stronger assumptions (namely,
for uniform domains of locally finite perimeter with boundary satisfying the lower ADR condition).

\section{Preliminaries}\label{sect:prelim}

Throughout the paper we work in $\ree$, $n\ge 1$. $H^n$ will denote the $n$-dimensional Hausdorff measure.
We will work with closed sets $E\subset\ree$ in which case we write $\sigma:=H^n|_E$. We will also consider open sets $\Omega\subset\ree$, not necessarily connected unless otherwise specified. In such case we shall write $\sigma:=H^n|_{\pom}$.

\begin{definition}[\bf Rectifiability]\label{d1} 
A set $E \subset \ree$, $n\ge 1$, is called {\bf $n$-rectifiable} if there exist $n$-dimensional Lipschitz maps $f_i : \rn \to \ree$ such that 
\begin{equation}
H^n\Big(E \setminus \bigcup_i f_i(\rn)\Big) = 0.
\end{equation}
\end{definition}

We next introduce a notion that is weaker than the well-known lower ADR condition:

\begin{definition}[\bf Weak Lower ADR (WLADR)]\label{d2} \
Let $E \subset \ree$, $n\ge 1$,  be a closed set with locally finite $H^n$-measure. We say the surface measure $\sigma := H^n \mres{E}$ satisfies the {\bf Weak Lower Ahlfors-David regular condition (WLADR)} if $\sigma(E\setminus E_*)=0$ where $E_*$ is the relatively open set
\begin{equation}
\label{eq:local-lower-ADR}
E_*=\left\{ 
x\in E: \inf_{
\substack{y\in B(x,\rho)\cap E
\\
0<r<\rho
}} \frac{H^n(B(y,r)\cap E)}{r^n}>0, \ \mbox{for some }\rho>0
\right\}.
\end{equation}
\end{definition}

Let us recall that $E$ is lower ADR if there exists a constant $c>0$ such that $\sigma(B(x,r)\cap E) \ge cr^n$ for all $x \in E$ and $r \in (0,\diam(E))$. Note that this is clearly stronger than WLADR. Also, if $E$ satisfies the lower ADR condition ``locally for small scales'' (that is, if for every $R$ the lower ADR condition holds on $E\cap B(0,R)$, albeit with constants depending on $R$, for all $0<r<r_R$ for some $r_R<R$) then WLADR holds. The WLADR condition says that for $\sigma$-a.e.~$x\in E$ there exists a small ball $B_x$ center at $x$ and a constant $c_x$ such that the lower ADR condition holds for all balls $B'\subset B_x$ with constant $c_x$. 
This in particular allows us to deal with cusps where the lower ADR condition fails as the radius approaches 0 (see next remark). Let us finally observe that WLADR is strictly stronger than the set having positive lower density  $H^n$-a.e.

\begin{remark} \label{remark:Omega-alpha}
There are examples of ``nice'' domains whose surface measure satisfies the WLADR condition but the lower ADR and/or the lower ADR condition ``locally for small scales'' fail. Let $\om \in \ree$, $n \ge 2$, be the domain above the graph of the function $|\cdot|^\alpha$ with $\alpha \in (0,\infty) \setminus \{1\}$, that is , $\om_\alpha = \{ (x',x_{n+1}) \in \rn \times \re : x_{n+1} > |x'|^{\alpha}\}$. When $\alpha>1$ the lower ADR condition fails at $0$ since $\sigma(B(0,r)\cap \pom_\alpha)/r^n\to 0$ as $r\to\infty$. However, it is easy to see that the lower ADR condition ``locally for small scales'' and hence WLADR follows. For $\alpha<1$, there is a cusp at $0$, and  one can see that the lower ADR condition at small scales fails since $\sigma(B(0,r)\cap \pom_\alpha)/r^n\to 0$ as $r\to 0^+$. However, one can easily obtain the  $(\pom_\alpha)_*=\pom_\alpha\setminus \{0\}$ (recall the notation in Definition \ref{d2}) and hence the  WLADR condition holds. See Figure \ref{figure}.
\end{remark}

\begin{figure}[h!]
\centering
\begin{tikzpicture}
\begin{scope}[shift={(-3,0)}]
\begin{scope}
\clip (0,0) circle (.8cm);
\draw[ultra thick, red] (-2.5,2.5) arc (90:0:2.5cm)  (0,0) arc (180:100:2.5cm) --(-2.5,2.5);
\end{scope}

\fill[thick, gray, opacity=.5] (-2.5,2.5) arc (90:0:2.5cm)  (0,0) arc (180:100:2.5cm) --(-2.5,2.5);

\draw (0,0) circle (.5pt) node[below] {$0$};
\draw[gray, opacity=1] (0,0) circle (.8cm);
\draw[dashed, gray, opacity=1] (0,0)--($(0,0)!.8cm!(-1,-1)$) node[near start, left] {\small $r$};
\node at (0,1.9) {$\Omega_{\alpha}$};
\node[right, red] at (0.1,.2) {$B(0,r)\cap\partial\Omega_{\alpha}$};
\node[below] at (0,-1){$\Omega_\alpha$ when $\alpha<1$ for $r$ small};
\end{scope}
\begin{scope}[shift={(3,0)}]
 \begin{scope}
 \clip (0,0) circle (1.9cm);
 \draw[ultra thick, red, domain=-1.25:1.25,smooth] plot (\x,{(abs(\x))^4});
 \end{scope}
        \fill[gray,opacity=.5, domain=-1.25:1.25,smooth] plot (\x,{(abs(\x))^4});
        \node[right, red] at (1,.6) {$B(0,r)\cap\partial\Omega_{\alpha}$};
\node at (0,1.15) {$\Omega_{\alpha}$};
\draw (0,0) circle (1pt) node[below] {$0$};

\draw[gray, opacity=1]  ([shift=(-10:1.9cm)]0,0) arc (-10:190:1.9cm);

\draw[dashed, gray, opacity=1] (0,0)--([shift=(190:1.9cm)]0,0) node[midway, below] {$r$};

\node[below] at (0,-1){$\Omega_\alpha$ when $\alpha>1$ for $r$ large};
\end{scope}

\end{tikzpicture}
\caption{\ }\label{figure}
\end{figure}

Our next definition introduces a subset of the boundary of a set in the spirit of the measure theoretic boundary (see \cite[Section 5.8]{EG}) but, here we only look at the infinitesimal behavior from the ``interior''.

\begin{definition}[\bf Interior Measure Theoretic Boundary]\label{d3} 
Given a set $\Omega \subset \ree$, the {\bf Interior Measure Theoretic Boundary} $\partial_+\om$ is defined as 
\begin{equation}\label{e100}
\partial_+\om:=
\left\{
x\in \pom:\ \limsup_{r \to 0^{+}} \frac{|B(x,r) \cap \Omega|}{|B(x,r)|} > 0
\right\}
\end{equation}
\end{definition}

Let us note that if an open set $\Omega$ satisfies a local (interior) corkscrew condition at $x\in\pom$, that is, if  there is $0<r_x<\diam(\pom)$ and $0<c_x<1$ such that for every $0<r<r_x$ there exists $B(X_{B(x,r)}, c_x\,r)\subset B(x,r)\cap \Omega$ then clearly $x\in \partial_+\Omega$.

\begin{remark}
Consider the domains $\Omega_\alpha$ as in Remark \ref{remark:Omega-alpha}. If $\alpha>1$, $\Omega_\alpha$ does not have interior corkscrews (for very large scales, the domain is too narrow and one cannot insert a ball of comparable radius), but it does have interior corkscrews for small scales. Hence $\partial_+\Omega=\pom$.
When $\alpha<1$, one can see that $\partial_+\Omega=\pom\setminus\{0\}$: with the exception of $0$ there are interior corkscrews for small scales, but at $0$ not only corkscrews fail to exist but also the $\limsup$ becomes 0.
\end{remark}

\begin{definition}[\bf Truncated Cones]\label{d4} If $z=(z',z_{n+1}) \in \ree$ then we write $\Gamma_{h,\alpha}(z)$ for the open truncated cone with vertex at $z$, with axis $e_{n+1}$, in the direction $e_{n+1}$, with height $h>0$ and with aperture $\alpha\in (0,\pi)$, that is,
$$
\Gamma_{h,\alpha}(z):=\big\{y = (y', y_{n+1}):\ |y' - z'|<(y_{n+1}-z_{n+1})\,\tan(\alpha/2),\ y_{n+1}\in(z_{n+1}, z_{n+1}+h)\big\}.
$$
We will often suppress $\alpha$ as what will matter is that the aperture is some fixed positive number.  We will sometimes use the notation $\Gamma^+$ (in place of $\Gamma$) and $\Gamma^-$ for the truncated cones in the direction $e_{n+1}$ and $-e_{n+1}$ respectively, this will only be necessary for the proof of Theorem \ref{T3}.
\end{definition}

The following result can be found in \cite[Theorem 15.11]{M} with the additional assumption that $H^n(E)<\infty$, however since the  $n$-linear approximability is a local property it immediately extends to any $E$ having locally finite $H^{n}$-measure.

\begin{theorem}[{$n$-linear approximability, \cite[Theorem 15.11]{M}}]\label{L1} Let $E \subset \ree$ be a $n$-rectifiable set such that $H^n \mres{E}$ is locally finite.   Then there exists $E_0\subset E$ with $H^n(E_0)=0$ such that if $x \in E\setminus E_0$ the following holds:  for every $\eta > 0$ there exist positive numbers $r_x=r_x(\eta)$ and $\lambda_x=\lambda_x(\eta)$ and a $n$-dimensional affine subspace $P_x=P_x(\eta)$ such that for all $0 < r < r_x$
\begin{equation}\label{e4}
H^n(E \cap B(y,\eta r))\ge \lambda_x r^n, \quad \text{for } y \in P_x\cap B(x,r)
\end{equation}
and
\begin{equation}\label{e5}
H^n\big((E \cap B(x,r))\setminus P_x^{(\eta r)}\big) < \eta r^n.
\end{equation}
Here $P_x^{(\eta r)}$ is an $\eta r$-neighborhood of $P_x$, that is, $P_x^{(\eta r)} = \{y \in \ree : \dist(y,P_x) \le \eta r\}$.
\end{theorem}

 \begin{theorem}[Dahlberg's Theorem, \cite{D1}]\label{L3}
 Suppose $\om$ is a bounded Lipschitz domain with surface measure $\sigma:= H^n\mres{\pom}$ then 
the harmonic measure associated to $\om$, $\hm$, is in $A_\infty(d\sigma)$. In particular, harmonic measure and surface measure are mutually absolutely continuous.
 \end{theorem}

\section{Proofs of the main Theorems}

We shall require two auxiliary lemmas.  As mentioned in the introduction, our arguments here
are similar in spirit to the proof of McMillan's Theorem as given in
\cite{GM}. 

\begin{lemma}[\bf Existence of Truncated Cones]\label{L2} Let $E \subset \ree$ be a $n$-rectifiable set with locally finite surface measure, write $\sigma:= H^n \mres{E}$ and use the notation in Theorem \ref{L1}. Given $x \in E\setminus E_0$ assume that there exists $\rho_x,c_x > 0$ such that 
\begin{equation}
\sigma(B(y,r)\cap E) \ge c_x\,r^n, 
\qquad 
\forall\,y \in B(x,\rho) \cap E, \ 0<r \le \rho_x.
\label{eq:LADR-aux}
\end{equation} 
For every $0<\eta<\eta_0(c_x):=\min\{2^{-4\,n}, c_x^2\}$, 
there exists a two sided truncated cone with vertex at $x$, height $h(\eta):=\eta^{\frac{1}{4n}}\,\min\{r_x(\eta), \rho_x\}$ and aperture $\alpha(\eta):=2\,\arctan\big(\eta^{-\frac{1}{4n}}/2)>\pi/2$ which does not meet $E$. (Note that $\alpha(\eta) \to \pi$ as $\eta \to 0^+$.)
\end{lemma}

We would like to call the reader's attention to the following fact. It is well-known that rectifiability is not affected by adding/removing sets with null $H^n$-measure. Hence we could augment $E$ by adding a countable dense set in $\ree$ and the resulting set will meet any truncated cone. This does not contradict the conclusion  of Lemma \ref{L2} since  \eqref{eq:LADR-aux} will not hold for the new set as it requires to have a lower ADR condition for all small balls near $x$ with the same constant $c_x.$ 

\begin{proof}
Without loss of generality we may take $x = 0$ and $P_x = \mathbb{R}^n \times \{0\}$. Given $0<\eta<\eta_0(c_x)$ we are going to see that $\Gamma_{h(\eta),\alpha(\eta)}(0)\subset\ree\setminus E$ with $h(\eta)$ and $\alpha(\eta)$ are in the statement.  Notice that $\alpha(\eta)> \pi/2$ from the choice of $\eta$ and also that $\alpha(\eta)\to \pi$ as $\eta \to 0^+$. A similar argument shows existence of similar truncated cone in the direction of $-e_{n+1}$.
Suppose (for the sake of contradiction) that there exist $r \in (0, h(\eta))$ and  $z \in D_r \cap E$ where
\begin{equation}
D_r = \left\{ z = (z', z_{n+1}) :\ z_{n+1} = \eta^{\frac{1}{4n}}r,\ |z'| <\frac{r}{2}\right\}.
\end{equation}
Then since $\eta_0\le 2^{-4n}$ it follows that 
\begin{equation}\label{e6}
E \cap B(z, \eta^{\frac{1}{2n}}r) \subset (E \cap B(0,r)) \setminus P_x^{(\eta r)}.
\end{equation}
On the other hand, the fact that $\eta_0\le c_{x}^2$ yields    
\begin{equation}\label{e7}
\sigma(E \cap B(z, \eta^{\frac{1}{2n}}r)) \ge c_{x}\eta^{\frac{1}{2}}r^n > \eta r^n,
\end{equation}
which together with \eqref{e6} contradicts \eqref{e5}. 
\end{proof}

\begin{lemma}[\bf Existence of Interior Truncated Cones]\label{R1}
Let $\om$ be an open set and whose $n$-rectifiable boundary, $\pom$, has locally finite surface measure $H^n|_{\pom}$. Assume that $x \in \pom\setminus(\pom)_*$ (recall the notation in Definition \ref{d2}) satisfies the hypothesis of Lemma \ref{L2} with $E = \pom$, $\rho_x$ and $c_x$ as above. Given $\epsilon > 0$ there exists $\tilde{\eta}_0 = \tilde{\eta}_0(\epsilon)<\eta_0(c_x)$  such that if $0<\eta < \tilde{\eta}_0$ and 
\begin{equation}\label{e101}
\limsup_{r \to 0^{+}} \frac{|B(x,r) \cap \om|}{|B(x,r)|} > \epsilon
\end{equation}
then one of the cones constructed in Lemma \ref{L2} must be in the interior of $\om$.
\end{lemma}

\begin{proof}
We may assume again that $x = 0$ and $P_x = \mathbb{R}^n \times \{0\}$ and  let $0<\eta< \tilde{\eta}_0 < \eta_0(c_x)$, where $\tilde{\eta}_0$ is to be chosen momentarily.
If $0<r < h(\eta)$ then by a rescaling argument
\begin{equation}\label{e102}
\frac{|B(x,r)\setminus (\Gamma^+_{h(\eta),\alpha(\eta)} \cup \Gamma^-_{h(\eta),\alpha(\eta)})|}{|B(x,r)|} 
= 
\frac{|B(x,h(\eta))\setminus (\Gamma^+_{h(\eta),\alpha(\eta)} \cup \Gamma^-_{h(\eta),\alpha(\eta)})|}{|B(x,h(\eta))|}
\end{equation}
and since $\alpha(\eta) \to \pi$ as $\eta \to 0^+$ one sees that
\begin{equation}\label{e103}
\frac{|B(x,h(\eta))\setminus (\Gamma^+_{h(\eta),\alpha(\eta)} \cup \Gamma^-_{h(\eta),\alpha(\eta)})|}{|B(x,r_0(\eta))|} \downarrow 0
\end{equation} as $\eta \to 0^+$. Choosing $\tilde{\eta}_0$ sufficiently small (depending on $\epsilon$) we have that 
\begin{equation}\label{e104}
\frac{|B(x,h(\eta))\setminus (\Gamma^+_{h(\eta),\alpha(\eta)} \cup \Gamma^-_{h(\eta),\alpha(\eta)})|}{|B(x,h(\eta))|} < \epsilon/2
\end{equation}
for any fixed $0<\eta<\tilde{\eta}_0$. On the other hand by \eqref{e101} there exists $0<r < h(\eta)$ such that 
\begin{equation}\label{e201}
\frac{|B(x,r) \cap \om|}{|B(x,r)|} > \epsilon.
\end{equation}
It follows from \eqref{e102}, \eqref{e104} and \eqref{e201} that at least one of the cones must meet $\om$. Recall that neither of the cones meet $E=\pom$, hence one of the cones must be interior.
\end{proof}

\begin{proof}[Proof of Theorem \ref{T2}] 
We show that $E$ being $n$-rectifiable implies \eqref{eq3} (the other implication is trivial).
Choose $\{\nu_m\}_{m=1}^{M}\subset \mathbb{S}^n$ (the unit sphere in $\ree$) such that for every $\nu \in \mathbb{S}^n$ there exists  $\nu_m$, $1\le m\le M$, such that $\ang(\nu,\nu_m) < \pi/8$. Set $P_m := \nu_m^\perp$, $1\le m\le M$.

Let us recall the definition of $E_*$ in \eqref{eq:local-lower-ADR} and note that for every $x\in E_*$ there exists $c_x, \rho_x>0$ such that
$$
\sigma(B(y,r)\cap E) \ge c_{x}\,r^n,
\qquad
\forall\, y\in B(x,\rho_x) \cap E, \ 0<r \le \rho_x.
$$
We use Theorem \ref{L1} and its notation. For every $k\in \N$ an $1\le m\le M$ we set
\begin{equation}\label{e8}
G(k, m):= \big\{x \in E_*\setminus E_0: \max\{c_x, \rho_x,r_x\} > 2^{-k}, \ \ang(P_m, P_x) < \pi/8 \big\}.
\end{equation}
Notice that setting $Z=(E\setminus E_*)\cup E_0$ we have that $\sigma(Z)=0$. Also, 
\begin{equation}
E=Z\cup \Big(\bigcup_{m=1}^M\bigcup_{k\in \N} G(k, m)\Big).
\label{eq:E-cover-Z}
\end{equation}
Hence, \eqref{eq3} follows at once if we show that each $G(k, m)$ can be covered by a countable union of boundaries of bounded Lipschitz domains missing $E$. 

Fix then $k\in \N$ and $1\le m\le M$ and we work with $G=G(k,m)$. By rotation, we may assume without loss of generality that $P_m = e_{n+1}^\perp$. Write $\eta_k:=\eta_0(2^{-k})$ (see Lemma \ref{L2}) and note that by Lemma \ref{L2} and the definition of $G_k$ it follows that if $0<\eta<\eta_k$ then for every $x\in G$ the  cone with vertex at $x$, axis $e_{n+1}$ (in the direction of $e_{n+1}$), aperture $\alpha(\eta)/2$ aperture and height $h(\eta)/2$ misses $E$.  At this stage we fix $0<\eta<\eta_k$ 1 and write $\Gamma_{h_0}= \Gamma_{h_0, \alpha_0}$ where $h_0=h(\eta)/2$ and $\alpha_0=\alpha(\eta)/2$.  What we have obtained so far is that $\Gamma_{h_0}(x)\subset \ree\setminus E$ for every $x\in G$. Now define the ``slices'', $S_{\ell}$, for $\ell \in \ZZ$, as follows 
\begin{equation}\label{e9}
S_\ell 
:= 
\left\{X \in \ree : X_{n+1} \in \left[\ell\,\frac{h_0}{10}, (\ell + 1)\frac{h_0}{10}\right) \right\}.
\end{equation}
Set $F_{\ell} := G \cap S_\ell$. Let $\pi_0$ be the projection of $\ree$ onto $\rn$ defined by $\pi_0(x) = \pi_0(x',x_{n+1}) = x'$. Now let $p_j \in \N$ be chosen so that the diameter of a $n$-dimensional cube of sidelength $2^{-p_j}$ is less than $\tfrac{h_0}{8}\tan(\alpha_0/2)$ and let $\dd_{p_j}$ be the collection of closed $n$-dimensional dyadic cubes with sidelength $2^{-p_j}$.
\begin{claim}\label{c1}
For every $Q \in \dd_{p_j}$ such that $\pi_0^{-1}(Q)\cap F_\ell \neq \emptyset$,
\begin{equation}\label{e105}
\om_{Q,\ell} 
:= 
\bigcup_{x \in \pi^{-1}(Q)\cap F_\ell} \Gamma_{h_0} (x)
\cap 
\left\{(z',z_{n+1}) \in \ree: z_{n+1} < (\ell +1) \frac{h_0}{10} + \frac{h_0}{2} \right\}
\end{equation}
is a bounded star-shaped domain with respect to a ball and hence a bounded Lipschitz domain.
\end{claim}
\begin{proof}[Proof of Claim \ref{c1}]
Without loss of generality we may assume $\ell = -1$. Let $y_Q$ be the center of $Q$ and set $Y_Q = (y_Q,\tfrac{h_0}{4})$. Take an arbitrary $x=(x', x_{n+1}) \in \pi_0^{-1}(Q)\cap F_\ell$. 
Since 
\begin{equation}\label{e106}
\Gamma_{\rm max} (x) := \Gamma_{h_0}(x) \cap \left\{(z',z_{n+1}) \in \ree: z_{n+1} < \frac{h_0}{2} \right\}
\end{equation} 
is convex, it suffices to show that $B(Y_Q,R) \subset \Gamma_{\rm max}(x)$ for some $R$ independent of $x$.  Note that  $ \Gamma_{\frac{h_0}{2}}(x',0) \subset \Gamma_{\rm max}(x)$, so we instead show that $B(Y_Q,R) \subset \Gamma_{\frac{h_0}{2}}(x',0)$ for some $R$ independent of $x'$. Recall that 
\begin{equation}\label{e107}
\Gamma_{\frac{h_0}{2}}(x',0) 
= 
\big\{ (z', z_{n+1}) : \ |z' - x'| < z_{n+1} \tan(\alpha_0/2), \ z_{n+1} \in (0,\tfrac{h_0}{2}) \big\}
\end{equation}
so that $K_{x'}=\{ (z', \tfrac{h_0}{4}) : |z' - x'| \le \tfrac{h_0}{8}\tan(\alpha_0/2)\}$ 
is a compact subset of $\Gamma_{\frac{h_0}{2}}(x',0)$. Set
\begin{equation}\label{e109}
R:= \frac{1}{2}\dist\big(K_{x'}, \partial \Gamma_{\frac{h_0}{2}}(x',0)\big)>0
\end{equation}
and
notice that $R$ has no dependence on $x'$. Also, by choice of $p_j$, we have that $|y_Q - x'| \le \tfrac{h_0}{8}\tan(\alpha_0/2)$.  Hence $Y_Q\in K_{x'}$ and $B(Y_Q,R) \subset \Gamma_{\frac{h_0}{2}}(x',0)  \subset \Gamma_{\rm max}(x)$ as desired. For a proof that bounded star-shaped domains with respect to a ball are bounded Lipschitz domains see \cite[Section 1.1.8]{Maz}.
\end{proof}

Once the claim is proved we observe that by construction, $\pi^{-1}(Q) \cap F_{\ell} \subset \pom_{Q,\ell}$, to see this we need only to observe that if $z_1, z_2 \in F_{\ell}$ then $z_2 \notin \Gamma_{h_0}(z_1)$, since $\Gamma_{h_0}(z_1)$ does not meet $E$. Then we have that 
\begin{equation}\label{G-Lip}
G = \bigcup_\ell F_\ell \subset \bigcup_\ell\bigcup_{Q \in \dd_{p_j}}\pom_{Q,\ell}
\end{equation}
where we take $\om_{Q, \ell} = \emptyset$ if $\pi^{-1}(Q)\cap F_\ell = \emptyset$. This completes the proof.
\end{proof}

\begin{proof}[Proof of Theorem \ref{T1}]
Let $\om_j$ be as in the statement of Theorem \ref{T2} and $F \subset E$  be such that $\sigma(F) > 0$. Then there exists $\om_j$ such that $\sigma(F \cap \pom_j) > 0$. Pick $X \in \om_j\subset\ree\setminus E$, let $\hm^X_{\om_j}$ be the harmonic measure for $\om_j$ with pole at $X$ and $\hm^X$ be the harmonic measure for $\ree\setminus E$ with pole at $X$. By the maximum principle and Dahlberg's Theorem (Theorem \ref{L3}) it follows that 
\begin{equation}\label{e12}
\hm^X(F) \ge \hm^X_{\om_j}(F \cap \pom_j) > 0,
\end{equation}
and the proof is complete.
\end{proof}

\begin{proof}[Proof of Theorem \ref{T3}]
We first show that $\pom$ being $n$-rectifiable implies \eqref{eq3:int} (the converse is trivial).
For every $x\in \pom$ we set
$$
\tau_x:=\limsup_{r \to 0^+} \frac{|B(x,r) \cap \om|}{|B(x,r)|}
$$
and recall that $\partial_+\Omega=\{x\in \pom:\ \tau_x>0\}$ and, by hypothesis, $\sigma(\pom\setminus \partial_+\Omega)=0$.

We follow the proof of Theorem \ref{T2} with $E=\pom$ with the following modifications. The set $G(k,m)$ is now defined as 
$$
G(k, m):= \big\{x \in (\pom)_*\setminus (\pom)_0: \max\{c_x, \rho_x,r_x, \tau_x\} > 2^{-k}, \ \ang(P_m, P_x) < \pi/8 \big\}
$$
so that \eqref{eq:E-cover-Z} holds where now 
$Z=\big(\pom\setminus ((\pom)_*\cap\partial_+\Omega)\big ) \cup (\pom)_0$ which again satisfies $\sigma(Z)=0$ (recall that $(\pom)_*$ and $(\pom)_0$ given respectively in Definition \ref{d2} and Theorem \ref{L1}). 
Again we just need to work with some fixed $G(k,m)$. Now we pick $\eta_k=\tilde{\eta}_0(2^{-k})$ (see Lemma \ref{R1}) and take $0<\eta<\eta_k$.  Next, we construct the domains as in the proof of Theorem \ref{T2}. Recall that in that construction we used the cones $\Gamma_{h_0}$ in the direction $e_{n+1}$. To emphasize this, let us write the cones as $\Gamma_{h_0}^+$ and also in \eqref{e105} we put $\om_{Q,\ell}^+$ in place of $\om_{Q,\ell}$. We know already all these are bounded Lipschitz domains.  We then may repeat the construction using the cones $\Gamma_{h_0}^-$ (in the direction $-e_{n+1}$) with the appropriate change  in \eqref{e106} and obtain that $\om_{Q,\ell}^-$ is another bounded Lipschitz domain. Note that by Lemma \ref{R1}, for every $x\in G(k,m)$, we have $\tau_x>2^{-k}$ and then either $\Gamma^+_{h_0}(x)$ or $\Gamma^-_{h_0}(x)$ is contained in $\Omega$. This implies that, if $\pi^{-1}(Q)\cap F_\ell \neq \emptyset$,  either $\om_{Q,\ell}^+$ or $\om_{Q,\ell}^-
 $ is contained in $\Omega$ (recall that both $\om_{Q,\ell}^+$ and $\om_{Q,\ell}^-$ connected domains that do not meet $\pom$) and we write $\om_{Q,\ell}^{\rm int}$ for the one that is contained in $\Omega$ (if both have this property we just pick one). As before, $\pi^{-1}(Q) \cap F_{\ell} \subset \pom_{Q,\ell}^{\rm int}$, to see this we observe that if $z_1, z_2 \in F_{\ell}$ then $z_2 \notin \Gamma_{h_0}^{\pm}(z_1)$ since $\Gamma_{c_2}^{\pm}(z_1)$ does not meet $E=\pom$. Then much as before
\begin{equation}\label{G-Lip-int}
G = \bigcup_\ell F_\ell \subset \bigcup_\ell\bigcup_{Q \in \dd_{p_j}}\pom_{Q,\ell}^{\rm int}
\end{equation}
where we take $\om_{Q, \ell}^{\rm int} = \emptyset$ if $\pi^{-1}(Q)\cap F_\ell = \emptyset$. This shows that $\pom$ can be covered by the boundaries of interior bounded Lipschitz domains, the proof that $\sigma \ll \hm$ is just as in the proof of Theorem \ref{T1}.
\end{proof}

\begin{remark}
From the proof of Theorem \ref{T2} one can see that given a closed set $E\subset \ree$  if we write $\E$ for the subset of $E$ containing all ``cone points'' ($x\in E$ is a cone point if there is a truncated open cone $\Gamma$ with vertex
at $x$, such that $\Gamma \subset \ree\setminus E$) then one has $\E\subset \cup_j \pom_j$ where 
the $\om_j$'s are bounded Lipschitz subdomains of $\ree \setminus E$. Hence, for any $F\subset \E$ with $H^n(F)>0$ there is $X\in\ree\setminus E$ for which $\omega^X(F)>0$. Note that the hypotheses in Theorem \ref{T1}, \ref{T2}, with the help of Lemma  \ref{L2},  guarantee that $\E$ has full $H^n$-measure on $E$. Analogously in the context of Theorem \ref{T3}, if we take the set of ``interior cone points'' (i.e., cone points whose associated cone is contained in $\Omega$) we can cover it by boundaries of bounded Lipschitz subdomains
contained in $\Omega$ and we get the corresponding absolute continuity. Again the hypotheses of Theorem \ref{T3} yield, after using Lemma \ref{R1}, that the  ``interior cone points'' have full $H^n$-measure on $\pom$.
\end{remark}

\section{Counterexamples}\label{sec:counter}

In this section we produce examples of domains with rectifiable boundaries for which surface measure fails to be absolutely continuous with respect to harmonic measure. The first example is a domain that does not have locally finite perimeter and the second one,  based on a construction presented by Jonas Azzam in January 2015 at ICMAT (Spain),  fails to satisfy the WLADR condition. These examples show that in Theorem \ref{T3} we cannot drop any of our background hypotheses. As we will observe below the same constructions allow us to obtain that in Theorem \ref{T3} we cannot also drop any of our background hypotheses.

 Let us point out that the assumption $\pom$ being $n$-rectifiable is also necessary by Theorem \ref{T5}. 
 We further note that in \cite[Section 4]{ABHM} there is an  example in $\re^3$ of a domain which is 1-sided NTA (in particular it is connected and  the Interior Measure Theoretic Boundary has full $H^n$-measure) and its boundary is ADR (hence it has locally finite $H^n$-measure and the WLADR condition holds). The boundary is not rectifiable (it is a cylindrical version of the ``4-corner Cantor set'' of J. Garnett) and surface measure is not absolutely continuous with respect to harmonic measure.  

In what follows, for a given domain $\Omega$ we will use the notation $\hm_\Omega^X$ for the harmonic measure for $\Omega$ with pole at $X\in\Omega$. In both examples we make use of the maximum principle, that is, if $X \in \om' \subset \om$ and $F \subset \pom$ then
\begin{equation}\label{e15}
\hm_{\om'}^X(F \cap \pom') \le \hm_{\om}^X(F).
\end{equation}

\begin{example}\label{ex:1}
 For $k \ge 1 $, and $n\ge 1$, set 
$$
\Sigma_k=\big\{(x,t)\in\reu:\ t = 2^{-k}, |x| \ge 2^{-k}\big\}
$$
and define
$$
\Omega:= \reu\setminus \left(\cup_{k=1}^\infty\Sigma_k \right),\qquad \Omega_k:= 
\reu\setminus \Sigma_k , \qquad \Omega_k':=\re^n\times(2^{-k},\infty) \,.$$
Then $\om$ is an open connected domain whose boundary clearly does not have locally finite $H^n$-measure (any surface ball centered at $\re^n\times\{0\}$ contains infinitely many $n$-dimensional balls of fixed radius). 
It is immediate to see that $\pom$ satisfies the WLADR condition as $(\pom)_*=\pom$ (recall the notation in Definition \ref{d2}). Notice also that $\Omega$ satisfies the interior Corkscrew condition (as the sets $\Sigma_k$ are located at heights which are separate enough) and hence $\partial_+\Omega=\pom$. Finally $\pom=(\re^n\times\{0\})\cup(\cup_{k=1}^\infty \Sigma_k$) which is $n$-rectifiable.

Take $X^*= (0, \dots, 0,2)\in\Omega$ and we are going show that $\hm^{X^*}_{\om}(F) = 0$ with $F=\re^n\times\{0\}\subset\pom$. Since $\om \subset \om_k$, \eqref{e15} implies that $\omega_\Omega^{X^*}(F)\le \omega_{\Omega_k}^{X^*}(F)$, hence we just need to see that $\omega_{\Omega_k}^{X^*}(F)\to 0$ as $k\to\infty$. Write $\Delta'_k = \{(x,t)\in\reu: t = 2^{-k}, |x| < 2^{-k}\}\subset \pom_k'$. Using the fact that harmonic measures for $\Omega_k$ and $\Omega_k'$ are probabilities, 
that $\Omega_k'\subset \Omega_k$ and maximum principle \eqref{e15} we see that
$$
\omega_{\Omega_k}^{X^*}(F)
=
1-\omega_{\Omega_k}^{X^*}(\partial\Omega_k\setminus F)
=
1-\omega_{\Omega_k}^{X^*}(\Sigma_k)
\le
1-\omega_{\Omega_k'}^{X^*}(\Sigma_k)
=
\omega_{\tilde{\Omega}_j}^{X^*}(\Delta_k').
$$
Since $\Omega_k'$ is a translation of $\reu$ we can use the classical Poisson kernel for the upper-half space $P(x,t)$ and one has that
\begin{equation}
\hm^{X^*}_{\om'_k}(\Delta'_k) = \int_{|y|<2^{-k}}P(y, 2-2^{-k})\,dy\to 0,
\qquad
\mbox{as }k\to\infty.
\end{equation}
This shows that surface measure  fails to be absolutely continuous with respect to harmonic measure $\omega_{\Omega}^{X^*}$ and hence with respect to $\omega_{\Omega}^{X}$ for every $X\in \Omega$ since $\Omega$ is connected.

To summarize, we have constructed $\Omega$, an open connected set, satisfying all the conditions in Theorem \ref{T3} with the exception that $\pom$ has locally finite $H^n$-measure, and  for which the conclusion of Theorem \ref{T3} fails.
\end{example}

\begin{remark}\label{remark:reflection}
If we repeat the same construction of Example \ref{ex:1} in the lower half-space and let $E$ be the boundary of the resulting open set (which has now 2 connected components), then  clearly $E$ satisfies all the hypotheses in Theorem \ref{T1}, except for $E$ having locally finite $H^n$-measure. In this case we can analogously prove that for the same set $F$ as before $\omega^X(F)=0$ for every $X\in\ree\setminus E$, hence the conclusion of Theorem \ref{T1} does not hold.
\end{remark}

\begin{example}  For $k\geq 1$, and $n\geq 2$, set
$$\Sigma_k:= \{(x,t)\in \reu:\, t=2^{-k},\, x\in \overline{\Delta(0,2^{-k}c_k)}+ c_k \ZZ^n\} \,,$$
where $c_k\downarrow 0$ will be chosen, and for $x\in \rn$, $\Delta(x,r):= \{y\in\rn:\, |x-y|<r\}$ is the 
usual $n$-disk of radius $r$ centered at $x$.   Define
$$\Omega:= \reu\setminus \left(\cup_{k=1}^\infty\Sigma_k \right),\qquad \Omega_k:= 
\reu\setminus \Sigma_k  \,,$$
which is clearly open and connected. Notice that $\Omega$ satisfies the interior
Corkscrew condition (note that the sets $\Sigma_k$ are located at heights which are separate enough),
hence $\partial_+\Omega=\pom$.

We assume that $c_k$ decays rapidly enough. It is easy to see that $\pom$ satisfies the upper ADR condition. Also, the WLADR (and hence the lower ADR) fails. To see this, given $X=(x,0)\in \pom$, we can find a 
sequence of balls $B_k=B(X_k, 2^{-k-2})$ with $X_k=(c_k\,\vec{l}_{k,x}, 2^{-k})\in \pom$ such that $X_k\to X$  with $\vec{l}_{k,x}\in\ZZ^n$. But then, for $k$ large enough $H^n(B_k\cap \pom)/(2^{-k-2})^n\approx c_k \to 0$ as $k\to \infty$. Hence 
$\re^n\times\{0\} \subset \pom\setminus (\pom)_*$  (recall the notation in Definition \ref{d2}) and the WLADR condition fails. 

\begin{remark}
Note that the argument that we have just used suggests possible relaxed version of the WLADR condition. To elaborate on this, first one can easily see that in the context of \eqref{eq:local-lower-ADR} one can alternatively write
\begin{align*}
E_{*}
&=
\Big\{ 
x\in E: \lim_{\rho\to 0^+} \inf_{
\substack{y\in B(x,\rho)\cap E
\\
0<r<\rho
}} r^{-n}\,H^n(B(y,r)\cap E)>0 
\Big\}
\\
&=
\Big\{ 
x\in E: \liminf_{B\to \{x\}} r_B^{-n}\,H^n(B\cap E)>0 
\Big\}
,
\end{align*}
where the $\liminf$ is taken over the balls $B=B(x_B,r_B),\, x_B\in E$, with $x_B\to x$ and $r_B\to 0$.
Consider the set $E_{**}$, where we replace above the $\inf$ with $\sup$ or, equivalently, $\liminf$ with $\limsup$.  Note that $E_*\subset E_{**}$. One might wonder whether the WLADR  condition can be replaced by the weaker fact $H^n(E\setminus E_{**})=0$.
In fact, this is not possible: in the current example with $E=\pom$, one can easily see that $E_{**}=E$. More generally, 
$H^n(E\setminus E_{**})=0$ for all $E \subset \ree$ $n$-rectifiable with $H^n \mres E$ being locally finite.  This follows from \cite[Theorem 16.2]{M}, which states that if  $E\subset \ree$ is $n$-rectifiable and $H^n \mres E$ is locally finite then the $n$-density
$$
\Theta^n(E,x) = \lim_{r \to 0^+}(2r)^{-n}H^n(B(x,r)\cap E)
$$
exists and is equal to $1$ for $H^n$ almost every $x \in E$.
(Note that \cite[Theorem 16.2]{M} is stated for  $H^n(E)<\infty$ but this may be easily  replaced by 
the condition that $H^n \mres E$ is locally finite.)

\end{remark}

Let $\hm^{(\cdot)}:= \hm^{(\cdot)}_\om$ and $\hm_k^{(\cdot)}:= \hm^{(\cdot)}_{\om_k}$ denote
harmonic measure for the domains $\om$ and $\om_k$ respectively.

\smallskip

\noindent{\bf Claim}.  If $c_k$ decays fast enough, then $\hm^{(\cdot)}(F) = 0$, with $F=\re^n\times\{0\}$.

Assuming this momentarily we have defined $\Omega$, an open connected set, satisfying all the conditions in Theorem \ref{T3} with the exception of the WLADR property, and  for which the conclusion of Theorem \ref{T3} fails. Again, as in Remark \ref{remark:reflection}, we may obtain a counterexample for Theorem \ref{T1} that satisfies all its hypotheses but the WLADR condition.

\smallskip

Before proving our claim we need to recall some definitions. Given $O\subset \ree$ an open and $K$  a compact subset of $O$ we define the capacity of  $K$ relative to $O$ as
$$
\mbox{cap}(K,O)=\inf\left\{\iint_O|\nabla\phi|^2\, dY:\ \phi\in C_0^\infty(O),
\ 
\phi\geq 1\mbox{ in }K\right\}.
$$
Also, the inhomogeneous capacity of $K$ is defined as 
$$
\mbox{Cap} (K)=\inf\left\{\iint_{\ree} \big(|\phi|^2+|\nabla\phi|^2\big)\, dY:\ \phi\in C_0^\infty(\re),
\ 
\phi\geq 1\mbox{ in }K\right\}.
$$
Combining \cite[Theorem 2.38]{HKM}, \cite[Theorem
 2.2.7]{AH} and \cite[Theorem 4.5.2] {AH} we have that if $K$ is a compact subset of $\overline{B}$, where $B$ is a ball with radius smaller than $1$, then
\begin{equation}\label{cap-dual}
\mbox{cap}(K,2B)
\gtrsim
\mbox{Cap}(K)
\gtrsim
\sup_{\mu} \frac{\mu(K)^2}{\|W(\mu)\|_{L^{1}(\mu)}},
\end{equation}
where the implicit constants depend only on $n$,
the sup runs over all Radon positive measures supported on $K$; and
$$
W(\mu)(X):=\int_0^1 \frac{\mu(B(X,t))}{t^{n-1}}\,\frac{dt}{t},
\qquad
X\in\supp \mu.
$$

\smallskip

We are now ready to prove our claim. We fix $k\geq 2$, take $c_k= 2^{-2kn}$ and write $N=N_k:=c_k^{-1}>1$.  
We are going to show that 
\begin{equation}\label{eq4.capbound}
\cp\big(\overline{B(X_0,s)}\cap\Sigma_k, B(X_0,2s)\big) \gtrsim s^{n-1}\,,\quad X_0:=(x_0,2^{-k})\in\Sigma_k,\ N^{-1/2}\le s<1\,.
\end{equation}
For a fixed $X_0$ and $s$, write $K=\overline{B(X_0,s)}\cap\Sigma_k$ and set $\mu = 2^{kn} s^{-1} \, H^n|_K\,,$
and note that for $X\in K$
\begin{equation}\label{eq4.mubound}
\mu\big(B(X,r)\big)\approx 2^{kn} s^{-1} \left\{
\begin{array}{l}
r^n\,, \,\,\qquad\,\, r<2^{-k} N^{-1}\,,
\\[4pt]
2^{-kn}N^{-n} \,,\,\, \,\,2^{-k} N^{-1}\leq r\leq N^{-1}
\\[4pt]
2^{-kn}r^n \,,\qquad  N^{-1}< r\leq s\\[4pt]
2^{-kn}s^n \,,\qquad r> s
\,.
\end{array}
\right.
\end{equation}
To compute $W(\mu)(X)$ for $X\in K$ write
$$W(\mu)(X) = \int_0^{2^{-k}N^{-1}}  +\, \int_{2^{-k}N^{-1}}^{N^{-1}}
+ \, \int_{N^{-1}}^s +\, \int_s^1\, =: \,I + II+III+IV\,.$$
Then, since $s\geq N^{-1/2}$,
$$I + II \lesssim 2^{kn} s^{-1} \left(2^{-k}N^{-1} \, + 2^{-kn}N^{-n}\int_{2^{-k}N^{-1}}^\infty\frac{dr}{r^n} \right)
\lesssim\, 2^{k(n-1)} N^{-1/2} \lesssim 1\,,$$
where the last bound holds by our choice of $N$ and $c_k$.
Furthermore, the last two estimates in \eqref{eq4.mubound} easily imply that $III+IV\lesssim 1$ and hence $W(\mu)(X)\lesssim 1$ for every $X\in K$. This, \eqref{cap-dual}, and \eqref{eq4.mubound} imply
as desired \eqref{eq4.capbound}:
$$
\cp\big(\overline{B(X_0,s)}\cap\Sigma_k, B(X_0,2s)\big) \gtrsim \mu(K)
\gtrsim
s^{n-1}.
$$

Set 
$$
P_k:= \left\{\left(x,2^{-k}-N^{-1/2}\right)\in\reu:\, x\in\rn\right\}\,,
$$
and observe that for $X\in P_k$,
$$
N^{-1/2}\le \delta_k(X):= \dist(X,\pom_k) = \dist(X,\Sigma_k)\le 2\, N^{-1/2}\,.$$
We now define
$$
u(X):= \hm_k^X(F) \,,\qquad X\in \om_k\,,
$$
and observe that $u\in W^{1,2}(\Omega_k)\cap C(\overline{\Omega_k})$ since $\pom_k$ is ADR (constants depend on $k$ but we just use this qualitatively) and $1_F$ is a Lipschitz function on $\pom_k$. Fix $Z_0\in P_k$ and let $Z_0'\in \Sigma_k$ be such that $|Z_0-Z_0'|=\dist(Z_0, \pom_k)\le 2\,N^{-1/2}$. Let $\Omega_{Z_0}=
\Omega_k\cap B(Z_0', \frac34 2^{-k})$, which is an open connected bounded set. 
We can now apply \cite[Example 2.12, Theorem 6.18]{HKM} to obtain $\alpha=\alpha(n)>0$  such that 
$$
u(Z_0) 
\lesssim 
\exp\left(-\alpha \int_{3\,N^{-1/2}}^{2^{-k-2}}  \frac{ds}{s} \right)  
\approx \big(2^k N^{-1/2}\big)^\alpha =2^{-\alpha k(n-1)}.
$$
where we have used since $u\equiv 0$ on $\pom_k \cap B(Z_0',2^{-k-1})$ and 
\eqref{eq4.capbound}.
Note that the last estimate holds for any $Z_0\in P_k$ and therefore, by the maximum principle,
$$u(x,t) \lesssim 2^{-\alpha k(n-1)}\,, \qquad (x,t) \in \om_k\,, \,\, t > 2^{-k}-N^{-1/2}\,.$$
In particular, if we set $X_0:= (0,\dots,0,1) \in \reu$, then by another application
of the maximum principle,
$$\hm^{X_0}(F)\, \leq \, \hm_k^{X_0}(F)=
u(X_0) \, \lesssim 2^{-\alpha k(n-1)}\, \to 0\,,$$
as $k \to \infty$, and the claim is established.
\end{example}

\appendix

\section{Local Theorems}\label{appendix}

We would like to point out that our method allows to obtain a local version of Theorem \ref{T3} (an analogous result can be proved for Theorem \ref{T1}, the precise statement is left to the interested reader). Let us recall that in Definition \ref{d2} we introduced the set $(\pom)_*$ (cf. \eqref{eq:local-lower-ADR} with $E=\pom$) formed by the points where the WLADR condition holds. Also, we remind the reader that $\partial_+\Omega$ denotes the  Interior Measure Theoretic Boundary as defined in \eqref{e100} in Definition \ref{d3}.

\begin{theorem}\label{THRM1}
Let $\om \subset \ree$, $n\ge 1$, be an open connected set whose boundary $\pom$ has locally finite $H^n$-measure. 
Suppose that $F \subset \pom$ is $H^n$-measurable, $n$-rectifiable and $H^n(F \setminus (\pom)_*) = H^n(F \setminus \partial_+\Omega) = 0$. Then
\begin{equation}
H^n \mres F \ll \hm \mres F \ll \hm.
\end{equation}
\end{theorem}

Additionally, by combining Theorem \ref{THRM1} with the results of \cite{AHM3TV}, we are able to decompose $\pom$ as a rectifiable portion, where surface measure is absolutely continuous with respect to harmonic measure, and a purely $n$-unrectifiable set with vanishing harmonic measure. 

\begin{theorem}\label{THRM2}
Let $\om \subset \ree$, $n\ge 1$, be an open connected set whose boundary $\pom$ has locally finite $H^n$-measure. 
Assume that $\pom$ satisfies the WLADR condition and that the Interior Measure Theoretic Boundary has full $H^n$-measure (i.e., $H^n(\pom \setminus (\pom)_*) = H^n(\pom \setminus \partial_+\Omega) = 0$). Then, there exists an $n$-rectifiable set $F \subset \pom$ such that $\sigma \mres F \ll \hm$, $\pom \setminus F$ is purely $n$-unrectifiable and $\hm(\pom \setminus F) = 0$. 
 \end{theorem}
 
To prove Theorem \ref{THRM2} (assuming Theorem \ref{THRM1}) note that, by the Lebesgue Decomposition Theorem (cf. \cite[page 42]{EG}), there exists a Borel set $F\subset \pom$, such that 
 \begin{equation}
 \sigma = \sigma_{\rm ac} + \sigma_{\rm s} = \sigma \mres F + \sigma \mres{\pom\setminus F}
 \end{equation}
with $\sigma_{\rm ac} \ll \hm$ and $\hm(\pom \setminus F) = 0$. By \cite[Theorem 1.1 (b)]{AHM3TV}, $F$ is rectifiable. It remains to show that $\pom \setminus F$ is purely $n$-unrectifiable. For the sake of a contradiction, suppose that $F'$ is a Borel $n$-rectifiable set such that $H^n(F' \cap (\pom\setminus F)) > 0$. Then, by Theorem \ref{THRM1} applied to the rectifiable set $F' \cap (\pom\setminus F)$, it follows that  $\hm(F' \cap (\pom \setminus F)) > 0$ which contradicts the fact that $\hm (\pom \setminus F ) = 0$.

\medskip

To prove Theorem \ref{THRM1} one can follow the argument in the proof of Theorem \ref{T3} with the following changes. First, it suffices to see that $F$ can be covered $H^n$-a.e. by a countable union of boundaries of Lipschitz domains contained in $\Omega$. To that end, we need to modify Theorem \ref{L1}, since we are only assuming that a piece of $\pom$ is $n$-rectifiable. We would like to emphasize that \eqref{e4} in Theorem \ref{L1} was never used in the arguments (the WLADR condition is a somehow stronger version of it) and hence we only need a version of \eqref{e5}. 

\begin{lemma}\label{LMA1} 
Suppose that $E \subset \ree$ is $H^n$-measurable with $H^n \mres E$ locally finite.
Let $F \subset E$ be an $H^n$-measurable and $n$-rectifiable set. Then there exists $F_0 \subset F$ with $H^n (F_0) = 0$ such that for every $x \in F\setminus F_0$ the following holds: 
for every $\eta > 0$ there exist positive numbers $r_x=r_x(\eta)$ and a $n$-dimensional affine subspace $P_x=P_x(\eta)$ such that for all $0 < r < r_x$
\begin{equation}\label{star}
H^n\big((E \cap B(x,r))\setminus P_x^{(\eta r)}\big) < \eta r^n.
\end{equation}
\end{lemma}

Assuming this result momentarily, one can easily establish versions of Lemmas \ref{L2} and \ref{R1} for $x\in F\setminus F_0$. With these in hand, we let $E=\pom$ and apply Lemma \ref{LMA1} to find $F_0$. Following the proof of Theorem \ref{T3}, the sets $G(k,m)$ need to be intersected with $F\setminus F_0$, and we also take 
$Z=\big(F\setminus ((\pom)_*\cap\pom_+)\big ) \cup F_0$. From this point the proof goes through 
\textit{mutatis mutandis}, details are left to the interested reader.

\begin{proof}[Proof of Lemma \ref{LMA1}]
Since $F$ is $H^n$-measurable and $n$-rectifiable with $H^n \mres F$ locally finite, we can apply Theorem \ref{L1} and find $F_0' \subset F$ with $H^n(F_0') = 0$ such that for every $x \in F\setminus F_0'$ the following holds:
for every $\eta > 0$ there exist positive numbers $r_1=r_1(x,\eta)$ and a $n$-dimensional affine subspace $P_x=P_x(\eta)$ such that for all $0 < r < r_1$
\begin{equation}\label{2-star}
H^n\big((F \cap B(x,r))\setminus P_x^{(\eta r)}\big)
\le
H^n\big((F \cap B(x,r))\setminus P_x^{(\eta r/2)}\big) < \eta r^n/2.
\end{equation}
Now by standard density estimates (see \cite[Theorem 6.2 (2)]{M}) we have that 
$$
H^n (F_0'')
:=
H^n \big(\big\{x\in F: 
\limsup_{r \to 0^+} H^n\big((E \setminus F) \cap B(x,r)\big)\,(2r)^{-n} >0\big\}\big)
=0.
$$
Hence, for every $x \in F\setminus F_0''$ the following holds: given $\eta > 0$ there exists $r_2 = r_2(x, \eta) > 0$ such that for all $0 < r < r_2$
\begin{equation}\label{4-star}
H^n\big((E \setminus F) \cap B(x,r)\big) < \eta r^n/2.
\end{equation}
Set $F_0=F_0'\cup F_0''$ which clearly satisfies $H^n(F_0) = 0$.  For every $x \in F\setminus F_0$, taking $P_x = P_x(\eta)$ as above, \eqref{2-star} and \eqref{4-star} give for every  $0 < r < r_x:=\min\{r_1,r_2\}$ 
\begin{multline*}
H^n\big((E \cap B(x,r))\setminus P_x^{(\eta r)}\big) \le H^n\big((F \cap B(x,r))\setminus P_x^{(\tilde{\eta} r)}\big)+ H^n\big((E \setminus F) \cap B(x,r)\big)< \eta r^n,
\end{multline*}
which proves \eqref{star}.
\end{proof}

\end{document}